\documentclass[a4paper,11pt,leqno]{amsart}

\usepackage{amsmath}
\usepackage{amsthm}
\usepackage{amsfonts}
\usepackage{amssymb}

\theoremstyle{plain}
\newtheorem{theorem}{Theorem}[section]

\newtheorem{lemma}[theorem]{Lemma}

\newtheorem{mainth}{Theorem}
\newtheorem{maincoro}{Corollary}

\theoremstyle{definition}

\newtheorem*{ac}{Acknowledgement}
    
\theoremstyle{remark}

\newtheorem{ex}[theorem]{Example}

\numberwithin{equation}{section}

\newcommand{\diam}{\operatorname{diam}}

\newcommand{\bR}{\mathbb R}
\newcommand{\bC}{\mathbb C}
\newcommand{\bN}{\mathbb N}

\newcommand{\bB}{\mathbb{B}}
\newcommand{\B}{\mathbb{B}}

\newcommand{\loc}{{\operatorname{loc}}}

\newcommand{\dist}{{\operatorname{dist}\,}}

\newdimen\vintkern\vintkern11pt
\def\vint{-\kern-\vintkern\int}

\newcommand{\norm}[1]{\lVert #1 \rVert}

\newcommand{\bS}{\mathbb{S}}

\newcommand{\R}{\mathbb{R}}

\newcommand{\cF}{\mathcal{F}}

\begin{document} 

\title[Rescaling principle for isolated essential singularities]{Rescaling principle for isolated essential singularities of quasiregular mappings}

\date{\today}

\author{Y\^usuke Okuyama}
\address{Division of Mathematics, Kyoto Institute of Technology, Sakyo-ku, Kyoto 606-8585, Japan}
\email{okuyama@kit.ac.jp}
\thanks{Y. O. is partially supported by JSPS Grant-in-Aid for Young Scientists (B), 24740087.}

\author{Pekka Pankka}
\address{University of Helsinki, Department of Mathematics and Statistics (P.O. Box 68), FI-00014 University of Helsinki, Finland}
\email{pekka.pankka@helsinki.fi}
\thanks{P. P. is partially supported by the Academy of Finland project \#256228.}



\maketitle

\begin{abstract}
We establish a rescaling theorem for isolated essential singularities of quasiregular mappings. 
As a consequence we show that the class of closed manifolds receiving a quasiregular mapping from a punctured unit ball with an essential singularity at the origin is exactly the class of closed quasiregularly elliptic manifolds, that is, closed manifolds receiving 
a non-constant quasiregular mapping from a Euclidean space.
\end{abstract}

\section{Introduction}
\label{sec:intro}

A continuous mapping $f\colon M\to N$ between oriented Riemannian $n$-manifolds
is \emph{$K$-quasiregular} if $f$ belongs to the Sobolev space $W^{1,n}_\loc(M,N)$ 
and satisfies the distortion inequality
\[
\norm{Df}^n \le K J_f \quad \mathrm{a.e.},
\]
where $\norm{Df}$ is the operator norm and $J_f$ is 
the Jacobian determinant of the differential $Df$ of $f$.

The main result of this paper is the following rescaling theorem. 
We denote the open unit ball about the origin in $\bR^n$ by $\B^n$. 
We say that a quasiregular mapping $f$ from $\bB^n\setminus\{0\}$ to a closed and oriented
Riemannian $n$-manifold $N$ has \emph{an essential singularity at the origin} if the limit 
$\lim_{x\to 0}f(x)$ does not exist in $N$.

\begin{mainth}\label{thm:3_option}
Let $N$ be a closed and oriented Riemannian $n$-manifold, $n\ge 2$, and let $f\colon \B^n\setminus\{0\} \to N$ be a $K$-quasiregular mapping with an essential singularity at the origin, $K\ge 1$. Then there exist a non-constant $K$-quasiregular mapping $g\colon X \to N$, where $X$ is either $\R^n$ or $\R^n\setminus \{0\}$, and sequences $(x_k)$ and $(\rho_k)$ in $\B^n$ and $(0,\infty)$, respectively, such that $\lim_{k\to\infty}x_k=0$, $\lim_{k\to\infty}\rho_k=0$ and
\begin{gather*}  
\lim_{k\to\infty}f(x_k+\rho_k v)= g(v)
\end{gather*}  
locally uniformly on $X$.
\end{mainth}

Theorem \ref{thm:3_option} bears a close resemblance to Miniowitz's Zalcman lemma for quasiregular mappings; see Miniowitz \cite{Miniowitz82} and Zalcman \cite{Zalcman75}. It seems, however, that this version for isolated essential singularities has gone unnoticed in the quasiregular literature although 
the heuristic idea behind this rescaling principle is well known in the classical holomorphic case ($n=2$ and $K=1$); see, e.g., Bergweiler \cite{Bergweiler00} and Minda \cite{Minda85}.

Theorem \ref{thm:3_option} readily yields the following characterization of closed and oriented Riemannian manifolds receiving a quasiregular mapping with an isolated essential singularity.

\begin{mainth}\label{th:elliptic}
Let $N$ be a closed and oriented Riemannian $n$-manifold, $n\ge 2$.
If there exists a $K$-quasiregular mapping $f\colon\bB^n\setminus \{0\} \to N$ having an essential singularity at the origin, $K\ge 1$, then there exists a non-constant $K'$-quasiregular mapping $g\colon\R^n\to N$ satisfying $g(\R^n)\subset f(\bB^n\setminus \{0\})$.

Conversely, if there exists a non-constant $K$-quasiregular mapping $g\colon\R^n\to N$, 
$K\ge 1$, then there exists a $K'$-quasiregular mapping $f\colon\bB^n\setminus \{0\} \to N$ having an essential singularity at the origin such that $f(\bB^n\setminus \{0\})\subset g(\R^n)$.

Here $K'=K'(n,K)\ge 1$ depends only on $n$ and $K$, and $K'(2,K)=K$.
\end{mainth}

Having Theorem \ref{th:elliptic} at our disposal, 
we readily obtain ``big'' versions of 
Varopoulos's theorem \cite[pp.\ 146-147]{VaropoulosN:Anagog} and
the  Bonk--Heinonen theorem \cite[Theorem 1.1]{BH01},
which respectively give a bound of 
the fundamental group 
and the de Rham cohomology ring of a closed quasiregularly elliptic manifold. 
Recall that a connected and oriented Riemannian $n$-manifold $N$, $n\ge 2$, 
is called \emph{quasiregularly elliptic} 
if there exists a non-constant quasiregular mapping from $\R^n$ to $N$.

\begin{maincoro}
\label{cor:bounds}
Let $N$ be a closed, connected, and oriented Riemannian $n$-manifold, $n\ge 2$, 
with a $K$-quasiregular mapping $\bB^n\setminus \{0\} \to N$ 
having an essential singularity at the origin, $K\ge 1$. 
Then the fundamental group $\pi_1(N)$ of $N$ has polynomial growth of order at most $n$, 
and the de Rham cohomology ring $H^*(N)$ of $N$ satisfies
\begin{equation}
\label{eq:dim}
\dim H^*(N) := \sum_{k=0}^n \dim H^k(N) \le C,
\end{equation}
where $C=C(n,K)>0$ depends only on $n$ and $K$.
\end{maincoro}

Although the former half of Corollary \ref{cor:bounds},
the big Varopoulos theorem,
is well-known to the experts, we have been unable to find it in the literature. 
For a direct proof of the big Bonk--Heinonen theorem, i.e., the bound \eqref{eq:dim}, see \cite{Pankka06}.

We would also like to note that 
together with the Holopainen--Rickman Picard theorem for quasiregularly elliptic manifolds \cite{HR98}, we obtain a big Picard type theorem for quasiregular mappings into closed manifolds; see also \cite{HP05}.

\begin{maincoro}\label{cor:bP}
Let $N$ be a closed, oriented, and connected Riemannian $n$-manifold, $n\ge 2$,
and $f\colon \bB^n \setminus \{0\}\to N$ be
a $K$-quasiregular mapping with an essential singularity at the origin, $K\ge 1$. 
Then for every $x\in N$, except for at most $q-1$ points, it holds that
$\#f^{-1}(x)=\infty$, where $q=q(n,K)\in\bN$ depends only on $n$ and $K$.
\end{maincoro}

We conclude this introduction with an application of Theorem \ref{thm:3_option} 
to the Ahlfors five islands theorem; see, e.g., Bergweiler \cite{Bergweiler00} or Nevanlinna \cite[XII \S 7, \S 8]{Nevanlinna70} 
for a detailed discussion. 

Let $f$ be a {\itshape quasimeromorphic function} on a domain $U$ in
$\bS^2$, i.e, a quasiregular mapping from $U$ to $\bS^2$.
We say that $f$ has a \emph{simple island $\Omega$ over a Jordan domain $D'$ in $\bS^2$} 
if $\Omega$ is a subdomain in $U$ and is mapped univalently onto $D'$ by $f$.
The Ahlfors five islands theorem states that 
\emph{given five Jordan domains in $\bS^2$ with pair-wise disjoint closures, 
any non-constant quasimeromorphic function on $\bR^2$
has a simple island over one of these Jordan domains.}

\begin{maincoro}
Let $f$ be a quasimeromorphic function on $\bB^2\setminus\{0\}$ 
having an essential singularity at the origin. Then
given five Jordan domains $D_1,\ldots, D_5$ in $\bS^2$ with pair-wise disjoint closures,
$f$ has a simple island over one of $D_1,\ldots, D_5$.
\end{maincoro}

\begin{proof}
Applying Theorem \ref{thm:3_option} to $f$, we obtain
sequences $(x_k)$ and $(\rho_k)$, and a non-constant quasimeromorphic function
$g$ on $\bR^2\setminus\{0\}$,
where $f_k$ is the mapping $v\mapsto f(x_k+\rho_k v)$ and $g$ is 
the locally uniform limit of $(f_k)$, as in Theorem \ref{thm:3_option}. 
We may fix Jordan domains $D_1',\ldots, D_5'$ in $\bS^2$
satisfying $D_j\Subset D_j'$ for every $j\in\{1,2,3,4,5\}$ and
having pair-wise disjoint closures.
By the Ahlfors five islands theorem,
the quasimeromorphic function
$g\circ \exp$ on $\bR^2$ has a simple island $\tilde{\Omega}'$ 
over one of these Jordan domains, say $D_j'$. 
Hence $g$ has a simple island $\Omega'$ over $D_j'$.
By Rouch\'e's theorem, 
for every $k\in\bN$ large enough,
$f_k$ has a simple island $\Omega_k\Subset\Omega'$ over $D_j\Subset D_j'$. 
\end{proof}

\begin{ac}
 The second author would like to thank Kai Rajala and Jang-Mei Wu for discussions related to Miniowitz's theorem.
\end{ac}

\section{Preliminaries}
\label{sec:pre}

Let $\B^n(x,r)$ be the open ball in $\bR^n$ about $x\in\bR^n$ of radius $r>0$. 
Set $\B^n(r):=\B^n(0,r)$ for each $r>0$ and set $\B^n:=\B^n(1)$. 
The corresponding closed balls are denoted by $\bar\B^n(x,r)$, $\bar\B^n(r)$, 
and $\bar\B^n$, respectively. 

Let $M$ be an oriented Riemannian $n$-manifold, $n\ge 2$. We denote by $|x-y|$ the distance between $x$ and $y$ in $M$, and by $B(x,r)$ the Riemannian ball $\{y\in M\colon|x-y|<r\}$ about $x\in M$ of radius $r>0$ in $M$. Similarly, we denote by $\bar B(x,r)$ the corresponding closed ball about $x\in M$ of radius $r>0$.

By \cite[III.1.11]{RickmanS:Quam}, every $K$-quasiregular mapping from an open set $U\subset \R^n$ 
to $\R^n$ is locally $\alpha$-H\"older continuous with $\alpha=(1/K)^{1/(n-1)}$. We refer to \cite{RickmanS:Quam} and \cite{Iwaniec-Martin-book} for the Euclidean theory of quasiregular mappings.

For every $x\in M$, there exist $r>0$ and a $2$-bilipschitz chart $B(x,r) \to \R^n$. Thus every $K$-quasiregular mapping from an open set $U\subset \R^n$ to $M$ is locally $\beta$-H\"older continuous with $\beta = \beta(n,K)$ depending only on $n$ and $K$.

The local H\"older continuity plays a key role in the proof of
the following manifold version (\cite[Theorem 19.9.3]{Iwaniec-Martin-book})
of Miniowitz's Zalcman lemma \cite[\S 4]{Miniowitz82}. Recall that a family $\cF$ of $K$-quasiregular mappings from a domain $\Omega$ in $\bR^n$ to $M$ is \emph{normal at $a\in \Omega$} if $\cF$ is normal on some open neighborhood of $a$.
 
\begin{theorem}
\label{th:Miniowitz}
Let $\Omega$ be a domain in $\R^n$, and $N$ be a closed, connected, and oriented Riemannian $n$-manifold, $n\ge 2$. Let $\mathcal{F}$ be a family of $K$-quasiregular mappings from $\Omega$ to $N$, 
$K\ge 1$. If $\mathcal{F}$ is not normal at $a\in \Omega$, then there exist sequences $(x_j)$, $(\rho_j)$, and $(f_j)$ in $\Omega$, $(0,\infty)$, and $\mathcal{F}$, respectively, and a non-constant $K$-quasiregular mapping $g \colon \R^n\to N$ such that $\lim_{j \to \infty} x_j = a$, $\lim_{j\to\infty}\rho_j=0$ and 
\begin{gather*}
\lim_{j\to\infty}f_j(x_j+\rho_j v)=g(v)
\end{gather*}
locally uniformly on $\bR^n$.
\end{theorem}

\section{Proof of Theorem \ref{thm:3_option}}

We begin the proof of Theorem \ref{thm:3_option} by showing
a manifold version of a classical lemma on isolated essential singularities
due to Lehto and Virtanen \cite{LV57Fennicae}; see also 
Heinonen and Rossi \cite[Theorem 2.3]{HR90} and Gauld and Martin \cite{GauldMartin93}.

\begin{lemma}\label{th:HeinonenRossi}
Let $M$ be a closed and oriented Riemannian $n$-manifold, $n\ge 2$,
and $f \colon \bB^n \setminus \{0\}\to M$ be
a quasiregular mapping with an essential singularity at the origin. Then
\begin{gather}
\label{eq:bounded}
\limsup_{r\to 0}\diam f(\partial \B^n(r))>0.
\end{gather}
\end{lemma}

\begin{proof}
The proof follows the argument of Heinonen and Rossi in \cite{HR90}.

Since the origin is an isolated essential singularity of $f$, there exist sequences $(z_k)$ and $(w_j)$ in $\B^n\setminus\{0\}$ such that $\lim_{k\to\infty}z_k=\lim_{j\to\infty}w_j=0$ and that both limits
\begin{gather*}
 a:=\lim_{k\to\infty}f(z_k)\quad\text{and}\quad b:=\lim_{j\to\infty}f(w_j) 
\end{gather*} 
exist in $M$ and are distinct. 

We choose $R\in(0,|b-a|/4)$ so small that $B(a,4R)$ is $2$-bilipschitz equivalent to an open ball in $\bR^n$. Note that, in particular, $b\not\in\bar B(a,4R)$.

Suppose that \eqref{eq:bounded} does not hold.
Then there is $r_0>0$ such that, for every $r\in(0,r_0)$,
$\diam(f(\partial \B^n(r)))<R/2$.

We fix $k_1\in\bN$ so large that $|z_{k_1}|<r_0$ and that $|z_{k_1}-a|\le R/2$. 
Let $r_1:=|z_{k_1}|$. Since $\diam(f(\partial \B^n(r_1)))<R/2$, we have
\begin{gather}
 f(\partial \B^n(r_1))\subset B(a,R).\label{eq:inside}
\end{gather}
Since $f(w_j) \to b\not\in \bar B(a,4R)$ as $j \to \infty$, 
the continuity of $f$ implies the existence of the maximal element, 
say $r_2\in (0,r_1)$, of the subset
\begin{gather*}
 \{r\in(0,r_1] \colon f(\partial \B^n(r))\not\subset B(a,2R)\}.
\end{gather*}
By the maximality, $f(\partial \bB^n(r_2))\cap \partial B(a,2R)\ne \emptyset$. Let $c\in f(\partial \B^n(r_2))\cap\partial B(a,2R)$. 
Then $\diam(f(\partial \B^n(r_2)))<R/2$ and
\begin{gather}
 f(\partial \B^n(r_2))\subset B(c,R/2).\label{eq:outside}
\end{gather} 

We join $\partial \B^n(r_1)$ and $\partial \B^n(r_2)$ by a line segment 
 $\ell$, which is contained, except for the end points, in the ring domain 
\begin{gather*}
A(r_2,r_1):=\B^n(r_1)\setminus \bar\B^n(r_2).
\end{gather*} 
Then the path $f(\ell)$ in $M$ joins $f(\partial \B^n(r_1))$ and $f(\partial \B^n(r_2))$, and by \eqref{eq:inside}, \eqref{eq:outside} and the choice of $r_1$ and $r_2$, we may fix $y_0\in\ell$ such that
\begin{gather*}
f(y_0)\in B(a,2R)\setminus(\bar B(a,R)\cup\bar B(c,R/2)).
\end{gather*}
Since $B(a,4R)$ is bilipschitz equivalent to a Euclidean $n$-ball, $M\setminus (\bar B(a,R)\cup \bar B(c,R/2))$ is connected. Thus
$f(y_0)$ can be joined with $b\not\in \bar B(a,4R)$ by a path 
$\beta \colon [0,1]\to M$ 
such that $\beta([0,1])\cap f(\partial (A(r_2,r_1)))=\emptyset$.

Let $\alpha \colon I_0\to\bB^n\setminus\{0\}$, where $I_0=[0,1]$ or $[0,t_0)$ for some $t_0\in(0,1]$, be a maximal lift of $\beta$ under $f$ starting at $y_0=\alpha(0)\in A(r_2,r_1)$. 
If $I_0 = [0,1]$, then $f(\alpha(1))=b\not \in f(A(r_2,r_1))$, so $\alpha(1)\not\in A(r_2,r_1)$. If $I_0\ne [0,1]$, then $\dist(\alpha(t),\partial \bB^n\cup \{0\}) \to 0$ as $t\to t_0$. 
In both cases, 
$\beta(I_0)\cap f(\partial A(r_2,r_1))=f(\alpha(I_0)\cap\partial A(r_2,r_1)) \neq \emptyset$. This is a contradiction and \eqref{eq:bounded} holds.
\end{proof}

\begin{proof}[Proof of Theorem $\ref{thm:3_option}$]
Define a function $Q_f\colon \B^n\setminus\{0\}\to [0,\infty)$ by
\begin{gather*}
 Q_f(x) := \sup_{y,y'\in \B^n(x,|x|/2), y\neq y'}\frac{|f(y)-f(y')|}{|y-y'|^\beta},
\end{gather*}
where $\beta=\beta(n,K)$ is as in Section \ref{sec:pre}. Put
\begin{gather*}
 M_f:=\limsup_{x\to 0}Q_f(x)|x|^{\beta}.
\end{gather*}

Suppose first that 
\begin{gather}
\label{eq:infty}
M_f=\infty,
\end{gather}
or equivalently, that there exists a sequence $(y_k)$ in $\B^n\setminus\{0\}$ satisfying $y_k\to 0$ and $Q_f(y_k)|y_k|^{\beta} \to \infty$ as $k \to \infty$.

Fix $\delta\in (0,1)$ and, for each $k\in\bN$,
define a mapping $g_k \colon \B^n(1+\delta) \to N$ by 
\begin{gather*}
g_k(z):= f(y_k+\frac{|y_k|}{2}z).
\end{gather*}
By \eqref{eq:infty}, there exist sequences $(z_k)$ and $(w_k)$ in $\B^n$ satisfying
\begin{gather*}
\limsup_{k\to\infty}\frac{|g_k(z_k)-g_k(w_k)|}{|z_k-w_k|^{\beta}}
\ge\limsup_{k\to\infty}\frac{1}{2}Q_f(y_k)\left(\frac{|y_k|}{2}\right)^{\beta}=\infty.
\end{gather*} 
Hence the family $\{g_k \colon k\in\bN\}$ is not normal on $\B^n(1+\delta)$. Indeed, otherwise, there exists a locally uniform limit point of 
$\{g_k \colon k\in\bN\}$, which is $K$-quasiregular
on $\bB^n(1+\delta)$ but not $\beta$-H\"older continuous on $\bar \bB^n$. 
This is impossible.

By Theorem \ref{th:Miniowitz}, there exist sequences $(z_j)$, $(\rho_j)$ and $(k_j)$ in $\B^n(1+\delta)$, $(0,\infty)$, and $\bN$, respectively, and a non-constant $K$-quasiregular mapping $g\colon \R^n \to N$ such that
$\lim_{j\to\infty}\rho_j=0$, $\lim_{j\to\infty}k_j=\infty$, and 
\begin{gather*}
 \lim_{j\to\infty}g_{k_j}(z_j + \rho_j v)=g(v)
\end{gather*}
locally uniformly on $\R^n$. This completes the proof in this case.

Suppose next that 
\begin{gather}
\label{eq:less_infty}
M_f<\infty.
\end{gather}
Let $\{g_k:\B^n(e^k)\setminus\{0\}\to N;k\in\bN\}$ be the family of $K$-quasiregular mappings defined as
\begin{gather*}
 g_k(v):= f(e^{-k}v).
\end{gather*}
By \eqref{eq:less_infty}, we have for every $v \in \R^n\setminus \{0\}$,
\[
\limsup_{k\to\infty}\sup_{w\in \B^n(v,|v|/2), w\neq v}\frac{|g_k(w)-g_k(v)|}{|w-v|^{\beta}} \le M_f|v|^{-\beta}<\infty,
\]
so the family $\{g_k \colon k\in \bN\}$ is 
locally equicontinuous on $\bR^n\setminus\{0\}$.

By Lemma \ref{th:HeinonenRossi}, there exists a subsequence $(g_{k_j})$ of $(g_k)$ satisfying
\begin{gather}
\lim_{j \to \infty} \diam(g_{k_j}(\bar \B^n\setminus \B^n(e^{-1}))) > 0.
\label{eq:circle}
\end{gather}
By passing to a further subsequence of $(g_{k_j})$ if necessary, we may assume, by the Arzel\`a--Ascoli theorem, that $(g_{k_j})$ converges locally uniformly on $\bR^n\setminus\{0\}$ to a mapping $g \colon \R^n\setminus \{0\} \to N$. Since $(g_{k_j})$ is a sequence of $K$-quasiregular mappings, $g$ is $K$-quasiregular, and by \eqref{eq:circle}, non-constant. This completes the proof.
\end{proof}

\begin{ex}
\label{ex:1}
To see that both cases in the proof of Theorem \ref{thm:3_option} 
actually occur, we give two examples, which are similar to Examples 23 and 24 in \cite{Pankka08}.

For $M_f<\infty$, we may take the conformal mapping $f\colon \B^n\setminus \{0\} \to \bS^{n-1}\times \bS^1$, $x\mapsto (x/|x|, e^{-i \log |x|})$. Then $f$ is the composition $\psi \circ h$ of a conformal homeomorphism $h \colon \bB^n\setminus \{0\} \to \bS^{n-1}\times \R$, $x \mapsto (x/|x|, -\log |x|)$, and a locally isometric covering map $\psi \colon \bS^{n-1}\times \R \to \bS^{n-1}\times \bS^1$. Since $|h(sx)-h(ty)| \le |\log s - \log t| + |x-y|$ for all $s,t\in (0,1)$ and all $x,y\in \bS^{n-1}$, we easily observe that $M_f < \infty$.

For $M_f=\infty$, we construct $f\colon \B^n\setminus \{0\} \to \bS^n$ using the winding map $h \colon \bS^n \to \bS^n$,
\begin{gather*}
(x_1,\ldots, x_{n-2}, re^{i \theta}) \mapsto (x_1,\ldots, x_{n-2}, re^{i 3\theta}),
\end{gather*}
which is a quasiregular endomorphism on $\bS^n$; we identify $\R^{n+1}$ with $\R^{n-1}\times \bC$.

Let $\sigma \colon \bS^n \setminus \{e_{n+1}\}\to \R^n$ be the stereographic projection, and $S$ the lower hemisphere $\{(x_1,\ldots, x_{n+1}) \in \bS^n \colon x_{n+1}\le 0\}$ of $\bS^n$. We note that $h|\partial S$ is the identity.

Let $(r_k)$ be a sequence tending to $0$ in $(0,1/2)$ and put
$x_k := r_k e_1$ and $B_k := \B^n(x_k, r_k/2)$ for each $k\in\bN$. 
We may assume that balls $B_k$, $k\in\bN$, are mutually disjoint. For each $k\in\bN$, 
put $x'_k := \sigma^{-1}(x_k)$ and $B'_k := \sigma^{-1}(B_k)$ and
let $\rho_k$ be the the M\"obius transformation on $\bS^n$ defined as 

\begin{gather*}
\rho_k(y) = \begin{cases}
 \sigma^{-1}\circ \alpha_k \circ \sigma (y),& y\ne e_{n+1}\\
 e_{n+1},& y=e_{n+1}
\end{cases}
\end{gather*}
where $\alpha_k$ is the affine transformation $x \mapsto r_k^{-1}(x-x_k)$ on $\R^n$. Then
$\rho_k(B'_k) = S$ and $\rho_k(x'_k) = -e_{n+1}$, so
the mapping $f\colon \B^n\setminus \{0\} \to \bS^n$ defined by
\begin{gather*}
 f(x) = \begin{cases}
 \sigma^{-1}(x), & x\in (\B^n\setminus \{0\})\setminus\bigcup_{k=1}^\infty B_k \\
 (\rho_k^{-1} \circ h \circ \rho_k)\circ \sigma^{-1}(x), & x\in B_k
 \end{cases}
\end{gather*}
is quasiregular with the same distortion constant as $h$. 
We observe that, for every $k\in \bN$, there exists a unique $y_k\in B_k$ 
satisfying $f(y_k)=e_1$ and $|x_k-y_k|\le C r_k^2$, where $C>0$ is independent of $k$.
Since $f(x_k)=e_{n+1}$ for every $k\in \bN$, by a direct computation, 
there is $C'>0$ such that for every $k\in\bN$,
$Q_f(x_k) |x_k|^\beta \ge C'r_k^{-\beta}$, so $f$ satisfies $M_f=\infty$.
\end{ex}

\section{Proof of Theorem \ref{th:elliptic}}

Let $Z_n:\bR^n\to\bR^n\setminus\{0\}$, $n>2$, be the Zorich mapping
(see \cite{Zorich67} or \cite[I.3.3]{RickmanS:Quam} for the construction of $Z_n$),
which is $K_n$-quasiregular for some $K_n\ge 1$ and an analogue of the 
exponential function $Z_2:\bC\to\bC\setminus\{0\}$. 
The mapping $Z_2$ is $K_2$-quasiregular for $K_2=1$.
Set $K'=K\cdot K_n\ge 1$ for each $n\ge 2$ and each $K\ge 1$.

Let $N$ be a closed, connected, and oriented Riemannian $n$-manifold.

Suppose there exists a $K$-quasiregular mapping $f\colon \bB^n \setminus \{0\} \to N$ with an essential singularity at the origin.
By Theorem \ref{thm:3_option} and a manifold version of Hurwitz's theorem (cf.\ the proof of \cite[Lemma 2]{Miniowitz82}), there exists a non-constant $K$-quasiregular mapping $g\colon X \to N$, where $X$ is either $\R^n$ or $\R^n\setminus \{0\}$, satisfying $g(X)\subset f(\bB^n\setminus \{0\})$.
If $X=\R^n$, then the mapping $g$ has the desired properties. If $X=\R^n\setminus\{0\}$, then
the mapping $g\circ Z_n \colon \R^n \to N$ 
has the desired properties.
 
Suppose now that $g\colon \R^n \to N$ is a non-constant $K$-quasiregular mapping. 
Let $\iota$ be an orientation preserving conformal involution of $\R^n\setminus \{0\}$ satisfying $\iota(\bB^n\setminus\{0\})=\bR^n\setminus\overline{\bB^n}$. Then the mapping $f \colon \bB^n\setminus \{0\}\to N$, $x\mapsto g\circ Z_n(\iota(x))$, has the desired properties.


\def\cprime{$'$}

\end{document}